\documentclass[11pt, reqno]{amsart} 

\title{On Borel Anosov subgroups of $\SL(d,\R)$}

\author{Subhadip Dey}
\address{Department of Mathematics, Yale University, New Haven, CT 06511}
\email{subhadip.dey@yale.edu}
\date{May 18, 2024}
     
\usepackage{graphicx}
\usepackage{mathtools,thmtools,enumitem}

\usepackage[linktocpage,breaklinks=true]{hyperref}
\hypersetup{
    colorlinks=true,
    linkcolor=blue,
    hypertexnames=false,
    filecolor=black,      
    urlcolor=black,
    citecolor=blue
}

\newcommand{\diagdots}[3][-25]{
  \rotatebox{#1}{\makebox[0pt]{\makebox[#2]{\xleaders\hbox{$\cdot$\hskip#3}\hfill\kern0pt}}}
}

\usepackage{amssymb}
\usepackage{calrsfs}

\usepackage{cleveref}
\newtheorem{ntheorem}{Theorem}

\newtheorem{ncorollary}[ntheorem]{Corollary}
\newtheorem{theorem}{Theorem}[section]
\newtheorem{proposition}[theorem]{Proposition}
\newtheorem{lemma}[theorem]{Lemma}

\newtheorem*{claim}{Claim}
\newtheorem*{claimone}{Claim 1}
\newtheorem*{claimtwo}{Claim 2}
\theoremstyle{remark}
\newtheorem{remark}[theorem]{Remark}
\newtheorem{example}[theorem]{Example}
\theoremstyle{definition}
\newtheorem{definition}[theorem]{Definition}

\DeclareMathOperator{\PSL}{PSL}
\DeclareMathOperator{\SL}{SL}
\DeclareMathOperator{\SO}{SO}
\DeclareMathOperator{\Sp}{Sp}

\def\C{\mathcal{C}}
\def\FF{\mathbf{F}}

\def\N{\mathbf{N}}
\def\R{\mathbf{R}}
\def\Z{\mathbf{Z}}
\def\GG{\mathfrak{G}}

\def\F{\mathcal{F}}
\def\acts{\curvearrowright}

\def\E{\mathcal{E}}

\def\G{\Gamma}
\def\g{\gamma}
\def\s{\hat\sigma}

\begin{document}

\maketitle

\begin{center}
{\em Dedicated to Misha Kapovich on the occasion of his 60\textsuperscript{th} birthday
}
\end{center}

\begin{abstract}
 We study the antipodal subsets of the full flag manifolds $\F(\R^d)$.
 As a consequence, for natural numbers $d \ge 2$ such that $d\ne 5$ or $d \not\equiv 0,\pm1 \mod 8$, we show that Borel Anosov subgroups of $\SL(d,\R)$ are virtually isomorphic to either a free group or the fundamental group of a closed hyperbolic surface.
 This gives a partial answer to a question asked by Andr\'es Sambarino.
  Furthermore,
  we show restrictions on the hyperbolic spaces admitting uniformly regular quasi-isometric embeddings into the symmetric space $X_d$ of $\SL(d,\R)$.
\end{abstract}

\section{Introduction}

In the past decade, {Anosov subgroups} of higher rank Lie groups have emerged as a well-regarded higher rank extension of the classical {convex-cocompact Kleinian groups}.
The notion of Anosov representations was introduced  by Labourie \cite{MR2221137} from a dynamical perspective in his pioneering work on {Hitchin representations} of surface groups
and then extended by Guichard--Wienhard \cite{MR2981818} for any {hyperbolic groups}.
Afterward, Kapovich--Leeb--Porti \cite{MR3736790} gave several geometrical and dynamical characterizations of Anosov subgroups; see the article by Kapovich--Leeb \cite{MR3888689}, giving an overview of their characterizations.
A main feature of Anosov subgroups is that they have a well-defined limit set in suitable generalized flag varieties, and any two distinct points in these limit sets are in general position.

This paper is motivated by a question asked by Andr\'es Sambarino, namely, whether Borel Anosov subgroups of $\SL(d,\R)$ are necessarily virtually free or  surface groups.
Combined works of Canary--Tsouvalas \cite{MR4186136} and Tsouvalas \cite{MR4184577} have affirmatively answered this question for $d = 3,4$, and $d \equiv 2\, \mod 4$ (note that $d=2$ case is classical). 
Using a different approach, we give an affirmative answer to this question  for all $d\in\N$ satisfying 
\begin{equation}\label{eqn:d}
 d\ne 5
 \quad\text{and}\quad
d \equiv 2,3,4,5, \text{ or } 6 \mod 8.
\end{equation}
See \Cref{cor:BA} below.

\medskip
We summarize our main objectives in this paper:
 
\begin{enumerate}
 \item We study the subsets of full flag manifolds $\F (\R^d)$ whose all pairs points are {\em antipodal}, i.e., are in a general position.
 As noted above, the limit sets of Anosov subgroups share this property.
 We are specifically interested to understand when antipodal subsets of $\F (\R^d)$ are {\em maximally} antipodal.
 See \S\ref{subsec:anosov} for discussions related to this matter.

 \item We aim to understand which hyperbolic groups can be realized as {\em Borel Anosov subgroups} of $\SL(d,\R)$. 
 See \S\ref{subsec:Borel} for the discussion related to this.
 
 \item Finally, we aim to understand which geodesic metric spaces may admit {\em coarsely uniformly regular quasi-isometric embeddings}, a notion introduced by Kapovich--Leeb--Porti \cite{MR3888689} strengthening the classical notion of {\em quasi-isometric embeddings}, into the symmetric space $X_d$ of $\SL(d,\R)$.
 Notably, the {\em orbit maps} of Anosov subgroups are such embeddings.
 See \S\ref{subsec:URU} for further discussions.
\end{enumerate}

\subsection{Antipodal subsets}\label{subsec:anosov}

For $d\ge 2$, let $\F_d \coloneqq \F(\R^d)$ denote the manifold consisting of all complete flags in $\R^d$.
A pair of points $\sigma_\pm\in\F_d$ is called {\em antipodal} (or {\em transverse}) if
\[
 \sigma_-^{(k)} + \sigma_+^{(d-k)} = \R^d,\quad
 \forall k\in\{1,\dots,d-1\}.
\]
In the above, for $\sigma\in\F_d$, we use the notation $\sigma^{(k)}$ to denote the $k$-dimensional vector subspace of $\R^d$ appearing in the complete flag $\sigma$.
We denote by $\E_\sigma$, $\sigma\in\F_d$, the set of all points in $\F_d$ which are {\em not} antipodal to $\sigma$.
The complementary subset of $\E_\sigma$ in $\F_d$ which we denote by $\C_\sigma$ is an open dense subset of $\F_d$ homeomorphic to a cell.
The subset $\C_\sigma$ is called a {\em maximal Schubert cell} or {\em big cell}, whereas $\E_\sigma$ is the closure of the union of all co-dimension one Schubert cells in the Schubert cell decomposition of $\F_d$ corresponding to $\sigma$.

\begin{ntheorem}\label{thm:main}
Let $d$ be any natural number satisfying \eqref{eqn:d}.
Let $\sigma_\pm\in \F_d$ be any pair of {antipodal} points, and let $\Omega$ be any connected component of
$
 \F_d \setminus (\E_{\sigma_-} \cup \E_{\sigma_+}) = \C_{\sigma_-} \cap \C_{\sigma_+}$.
 If $c : [-1,1] \to \F_d$ is any continuous map such that
 \[
  c(\pm 1) = \sigma_\pm \quad \text{and} \quad
  c((-1,1)) \subset \Omega,
 \]
 then, for all $\sigma\in \Omega$, the image of $c$ intersects $\E_\sigma$.
\end{ntheorem}

Although the following example is not covered in the setting of the theorem, we believe that it would still serve as a simple illustration of the statement:
In the case corresponding to $\SL(2,\R) \times \SL(2,\R)$ and its minimal parabolic subgroup, the ``full flag manifold'' is realized as a torus.
Let $D$ denote the unit square in $\R^2$ from which we obtained the torus by identifying the opposite edges.
We identify the four corners of $D$ with $\sigma_-$, and $\E_{\sigma_-}$ with its edges.
Given any point $\hat\sigma$ in the interior of $D$, the subset $\E_{\hat\sigma}$ can be realized as the union of the horizontal and vertical line segments passing through $\hat\sigma$.
Therefore, for any $\sigma_+ \in \C_{\sigma_-} = {\rm int}\, D$, the intersection $\C_{\sigma_-} \cap \C_{\sigma_+}$ can be seen as the disjoint union of four open rectangles. For any path $c$ connecting $\sigma_\pm$ lying in (except for the endpoints) one such rectangles $\Omega$, and for any point $\sigma \in \Omega$, it can be checked that $\E_\sigma$ intersects $c$.

We prove \Cref{thm:main} in \S\ref{sec:thm:main}.
The main technical ingredient in the proof of this result is \Cref{thm:swaps}, which states that, for natural numbers $d$ satisfying \eqref{eqn:d}, an involution $\iota$ defined on $\C_{\sigma_-} \cap \C_{\sigma_+}$
does not leave invariant any connected components.
See \S\ref{sec:heisen}.

\begin{remark}\label{rem:resd}
Let us highlight the reason why \Cref{thm:main} and the other key results below have the restriction on $d$ given by \eqref{eqn:d}: This is due to the fact that \Cref{thm:swaps} {possibly} fails when $d$ is of the form $8k-1$, $8k$, or $8k+1$, for $k\in\N$. Specifically, we show that \Cref{thm:swaps} indeed fails when $d = 8k\pm1$, yet the status of its validity remains unclear for $d = 8k$.
However, with help from Su Ji Hong, we could computationally verify the validity of \Cref{thm:swaps} when $d=5$.
Further discussions on this matter are detailed in \Cref{rem:swaps}.
Consequently, for $d=5$, Theorems \ref{thm:main}, \ref{thm:main0}, \ref{thm:main1}, and Corollaries \ref{cor:BA}, \ref{cor:URQI} (as well as \Cref{cor:BAothersplit} for $n=2$ in (i)) remain valid.

Nevertheless, it is an intriguing prospect to investigate whether \Cref{thm:main} holds true for these remaining natural numbers $d$, provided the hypothesis is strengthened by requiring the map $c : [-1,1] \to \F_d$ to also be antipodal.
\end{remark}

We apply \Cref{thm:main} to get information about {\em (locally) maximally antipodal subsets} of $\F_d$, defined as follows.

\begin{definition}[Antipodal subsets and maps]\leavevmode
\begin{enumerate}[label=(\roman*)]
 \item  A subset $\Lambda\subset \F_d$ is called {\em antipodal} if all distinct pairs of points in $\Lambda$ are antipodal.
 
 \item
An antipodal subset $\Lambda\subset \F_d$ is called {\em maximally} antipodal if it is not contained in a strictly larger antipodal subset of $\F_d$.

 \item
 We call an antipodal subset $\Lambda\subset \F_d$ {\em locally} maximally antipodal if there exists an open neighborhood $N$ of $\Lambda$ in $\F_d$ such that $\Lambda$ is not contained in any strictly larger antipodal subset of $N$; equivalently, every point of $N$ is {\em not} antipodal to some point of $\Lambda$.
 \item
A continuous map $\phi: Z \to\F_d$ is called {\em antipodal} if for all distinct points $z_\pm\in Z$, $\phi(z_+)$ and $\phi(z_-)$ is a pair of antipodal points.
\end{enumerate}
\end{definition}

Note that antipodal subsets of $\F_d$ form a {poset}, partially ordered by inclusions, and the maximally antipodal subsets are precisely the maximal elements.

As an application of \Cref{thm:main}, we get the following result.

\begin{ncorollary}\label{cor:maximal}
Let $d$ be any natural number satisfying \eqref{eqn:d}.
 If $c: S^1\to \F_d$ is an antipodal embedding,
 then $\Lambda \coloneqq c(S^1)$ is a locally maximally antipodal subset of $\F_d$.
\end{ncorollary}

\begin{proof}
 Let $x_1,x_2,x_3 \in S^1$ be any distinct triple.
 For distinct indices $i,j,k\in \{1,2,3\}$, let $\Omega_{ijk}$ denote the connected component of $\F_d \setminus (\E_{c(x_i)} \cup \E_{c(x_k)})$ containing $c(x_j)$.
 Then, $Y = \Omega_{123}\cup \Omega_{231} \cup \Omega_{312}$ is an open neighborhood of $\Lambda$.
 Applying \Cref{thm:main}, one can verify that every point in $Y$ is non-antipodal to some point in $\Lambda$.
\end{proof}

It is unclear whether one can omit the word ``locally'' in the conclusion of the above result.
However, if the image of $c: S^1\to \F_d$ is the limit set of a Borel Anosov subgroup of $\SL(d,\R)$, then 
$\Lambda \coloneqq c(S^1)$ is a maximally antipodal subset of $\F_d$.
See \Cref{prop:remone}.

\subsection{Borel Anosov subgroups}\label{subsec:Borel}

First, let us recall the notion of {\em boundary embedded subgroups} of $\SL(d,\R)$ introduced by Kapovich, Leeb, and Porti \cite{MR3736790}. 

\begin{definition}[Boundary embedded subgroups]\label{def:BE}
A subgroup $\G$ of $\SL(d,\R)$ is called {\em $B$-boundary embedded} if $\G$, as an abstract group, is hyperbolic and there exists a $\G$-equivariant antipodal embedding $\xi:\partial_\infty\G \to \F_d$ of the Gromov boundary $\partial_\infty\G$ of $\G$ to the complete flag manifold $\F_d$.
\end{definition}

Due to the fact that non-elementary hyperbolic groups act as convergence groups on their Gromov boundaries, it can be inferred that non-elementary $B$-boundary embedded subgroups of $\SL(d,\R)$ are discrete. 
The following result shows that the group theoretic structures of the $B$-boundary embedded subgroups of $\SL(d,\R)$ are highly restricted.

\begin{ntheorem}\label{thm:main0}
Let $d$ be any natural number satisfying \eqref{eqn:d}.
If a subgroup $\G$ of $\SL(d,\R)$ is $B$-boundary embedded, then $\G$ is virtually isomorphic to either a free group or the fundamental group of a closed hyperbolic surface.
\end{ntheorem}

This result, which we prove in \S\ref{proofs:mainresults}, directly applies to the class of Borel Anosov subgroups introduced by Labourie \cite{MR2221137}. Labourie proved the seminal result that the images of the \textit{Hitchin representations} of surface groups into $\SL(d,\R)$ are Borel Anosov subgroups. While Labourie's original definition of Borel Anosov subgroups was intricate, a more straightforward definition has since emerged thanks to the work of Kapovich–Leeb–Porti \cite{MR3736790,MR3890767} and Bochi--Potrie--Sambarino \cite{bochi2019anosov}.

For $g\in \SL(d,\R)$, let
\begin{equation*}
\sigma_1(g) \ge \dots \ge \sigma_d(g),
\end{equation*}
denote the singular values of $g$.
 For a finitely generated group $\G$, let $|\cdot|:\G \to \N \cup\{0\}$ denote the word-length function with respect to some symmetric finite generating set of $\G$.
The following definition does not depend on the choice of such a generating set, although the implied constants may vary.

\begin{definition}[Borel Anosov subgroups]\label{def:results_anosov}
 A finitely generated subgroup $\G$ of $\SL(d,\R)$ is called {\em Borel Anosov} if there exist constants $L\ge 1$ and $A\ge 0$ such that, for all $k\in \{1,\dots,d-1\}$ and for all $\g\in\G$,
\begin{equation}\label{eqn:results_anosov}
  \log \left(\frac{\sigma_{k}(\g)}{\sigma_{k+1}(\g)}\right) \ge L^{-1}|\g| - A.
\end{equation}
\end{definition}

The main features of the Borel Anosov subgroups $\G$ of $\SL(d,\R)$ includes (i) $\G$, as an abstract group, is hyperbolic and (ii) there exists a $\G$-equivariant antipodal embedding, called the {\em limit map},
\[
 \xi : \partial_\infty \G \to \F_d,
\]
from the Gromov boundary $\partial_\infty \G$ of $\G$ to the complete flag manifold $\F_d$.
See \cite{MR2221137,bochi2019anosov}.
In particular, {Borel Anosov subgroups} of $\SL(d,\R)$ are $B$-boundary embedded (\Cref{def:BE}).
Therefore, \Cref{thm:main0} has the following direct implication.

\begin{ncorollary}\label{cor:BA}
Let $d$ be any natural number satisfying \eqref{eqn:d}.
If $\G$ is a Borel Anosov subgroup of $\SL(d,\R)$, then $\G$ is virtually isomorphic to either a free group or the fundamental group of a closed hyperbolic surface.
\end{ncorollary}

\begin{remark}
This result partially answers a question asked by Andr\'es Sambarino (see \cite[\S 7]{MR4186136}) who asked if the statement is true for all $d\ge 2$. 
As mentioned above, this question previously has been affirmatively answered for $d = 3$ and $d=4$ by Canary--Tsouvalas \cite{MR4186136}, and for all $d$ of the form $4k+2$ by Tsouvalas \cite{MR4184577}.
In fact, in this article, we give a new (and possibly simpler) proof for the previously known cases from \cite{MR4186136,MR4184577}. 
However, for the remaining integers $d\ge 2$ not covered by \Cref{cor:BA} (except for $d=5$), we are unable to provide a conclusive answer to this question. Cf. \Cref{rem:resd}. 
We emphasize a connection between the maximal antipodality of limit sets and Andr\'es Sambarino's question, which could be beneficial for further exploration in these remaining cases:
Suppose there exists $d\in\N$ and a Borel Anosov subgroup $\G< \SL(d,\R)$, isomorphic to a surface group, such that the limit set of $\G$ in $\F_d$ is not maximally antipodal. In this case, by applying the  Combination Theorem for Anosov subgroups \cite{MR4002289} (see also \cite{Dey:2022uz}), one can construct a Borel Anosov subgroup of $\SL(d,\R)$ isomorphic to $\G' \star \Z$, where $\G'$ is a finite index subgroup (thus, a surface subgroup) of $\G$. Such a construction could produce a counter-example.
\end{remark}

More generally, given a connected, non-compact, real semisimple Lie group $G$ with finite center and a parabolic subgroup $P$ of $G$, there is a distinguished class of discrete subgroups of $G$ called {\em $P$-Anosov subgroups} \cite{MR2981818} (see also \cite{MR3736790}).
By a {\em $B$-Anosov subgroup} of $G$, we are referring to a $P$-Anosov subgroup, where $P$ is assumed to be a minimal parabolic subgroup of $G$.\footnote{When dealing with a connected algebraic group $G$ defined over an algebraically closed field, the minimal parabolic subgroups are Borel subgroups. This is why we refer this class of subgroups as ``$B$-Anosov.'' For the same reason, we referred  to the class of subgroups defined in \Cref{def:BE} as ``$B$-boundary embedded.'' When the minimal parabolic subgroups of $G$ are Borel (e.g., if $G$ is split),  we may also refer $B$-Anosov subgroups as {\em Borel Anosov subgroups} as done in \Cref{def:results_anosov} for the case $G= \SL(d,\R)$.} 

In the special case $G = \SO_0(n,n+1)$ (resp. $G = \Sp(2n,\R)$),  the Borel Anosov subgroups of $G$ map to Borel Anosov subgroups of $\SL(2n+1,\R)$ (resp. $\SL(2n,\R)$) under the inclusion $\SO_0(n,n+1) \hookrightarrow \SL(2n+1,\R)$ (resp. $\Sp(2n,\R) \hookrightarrow \SL(2n,\R)$).
Thus, \Cref{cor:BA} also yields the following:

\begin{ncorollary}\label{cor:BAothersplit}
 Let $G$ be one of following:
\begin{enumerate}[label=(\roman*)]
 \item $\SO_0(n,n+1)$, where $n\ne 2$ and $n \equiv 1$ or $2\mod 4$.
 \item $\Sp(2n,\R)$, where $n \equiv 1,2$ or $3\mod 4$.
\end{enumerate}
 If $\G$ is a Borel Anosov subgroup of $G$, then $\G$ is virtually isomorphic to either a free group or the fundamental group of a closed hyperbolic surface.
\end{ncorollary}

Finally, it is worth remarking that restrictions on Anosov subgroups of the symplectic groups are explored further in the subsequent papers \cite{dey2023restrictions} and \cite{pozzetti2023projective}.
 
\subsection{Uniformly regular quasi-isometric embeddings}\label{subsec:URU}

The notion of {\em uniformly regular quasi-isometric embeddings}, introduced by Kapovich--Leeb--Porti, of geodesic metric spaces into the symmetric space 
\[X_d \coloneqq \SL(d,\R)/\SO(d,\R)\]
is a strengthening of {quasi-isometric embeddings}.
Since the definition of uniformly regular quasi-isometric embeddings requires a lengthier discussion, we refer our reader to \cite[Definition 2.26]{MR3888689}.
This notion is especially interesting in the context of Anosov subgroups since, by \cite[Theorem 3.41]{MR3888689}, a subgroup $\G<\SL(d,\R)$ is Borel Anosov if and only if $\G$ is finitely-generated and the orbit map 
\begin{equation}\label{orbitmap}
  \G \to X_d, \quad \g \mapsto \g\cdot x_0,
\end{equation}
is a uniformly regular quasi-isometric embedding, where $\G$ is equipped with any word metric and $x_0 \in X_d$ is any base-point (cf. \eqref{eqn:results_anosov}).

\begin{ntheorem}\label{thm:main1}
Consider any locally compact, geodesic, Gromov hyperbolic space $Z$ with $\partial_\infty Z$ denoting its Gromov boundary. Suppose there exists a topological embedding $c: S^1\to \partial_\infty Z$ such that the image of $c$ is not an open set.
Then $Z$ does not admit any uniformly regular quasi-isometric embeddings into $X_d$, given $d$ satisfies \eqref{eqn:d}.
\end{ntheorem}

 \Cref{thm:main1} is proved in \S\ref{proofs:mainresults}. This result obstructs uniformly regular quasi-isometric embeddings of certain simply connected  complete Riemannian manifolds of non-positive sectional curvature (also called {\em Cartan–Hadamard manifolds}) into $X_d$: 
More precisely, if $Y$ is a Cartan–Hadamard manifold with  sectional curvature bounded below
and $Y$ admits a 
uniformly regular quasi-isometric embedding into $X_d$, then $Y$ is Gromov hyperbolic as a metric space  (by \cite[Theorem 1.2]{MR3890767}), whereas 
the Gromov boundary of $Y$ is homeomorphic to the sphere of dimension $\dim Y -1$ (by \cite[Theorem 2.10]{Kaimanovich}).
Thus, if $d$ satisfies \eqref{eqn:d}, then  \Cref{thm:main1} implies that
$\dim Y \le 2$. 
A special case of this is as follows:

\begin{ncorollary}\label{cor:URQI}
The hyperbolic plane is the only symmetric space of non-compact type that admits uniformly regular quasi-isometric embeddings into $X_d$ with $d$ satisfying  \eqref{eqn:d}.
\end{ncorollary}

In a similar vein, applying  \Cref{thm:main1}, a stronger conclusion than \Cref{cor:BA} can be obtained: 
finitely-generated groups $\G$, unless $\G$ is virtually a free group or a surface group, do not even admit uniformly regular quasi-isometric embeddings\footnote{Which need not arise from a group homomorphism $\G\to \SL(d,\R)$ as in \eqref{orbitmap}.} 
 into $X_d$ when $d$ satisfies \eqref{eqn:d}, 
since in such cases, if $\G$ admits a uniformly regular quasi-isometric embedding into $X_d$, then $\G$ is a hyperbolic group (\cite[Theorem 1.2]{MR3890767}) and there exist such non-isolated circles in $\partial_\infty\G$ (\cite[Corollary 2]{MR2146190}) as required by the hypothesis of \Cref{thm:main1} to get a contradiction.

\subsection*{Outline of this paper} 

In Section \ref{sec:heisen}, we present and prove the main technical result of this paper, Theorem \ref{thm:swaps}. Using this result, we establish Theorem \ref{thm:main} in Section \ref{sec:thm:main}. Subsequently, we apply Corollary \ref{cor:maximal}, an immediate consequence of Theorem \ref{thm:main}, to prove Theorems \ref{thm:main0} and \ref{thm:main1} in Section \ref{proofs:mainresults}. Finally, in Section \ref{sec:further}, we explore some additional applications of the methods introduced in this paper.

\subsection*{Acknowledgement} 
I extend my sincere thanks to Misha Kapovich and Yair Minsky for their  suggestions and encouragement.
I am grateful to Richard Canary, Su Ji Hong, Or Landesberg, and Max Riestenberg for the engaging discussions related to this work.
Special thanks to Hee Oh for her insightful question (referenced in \Cref{prop:oh}) and the  discussions stemming from it and to Misha Shapiro for very helpful discussions related to \Cref{thm:swaps}.
I express my gratitude to the referee for their careful review of this article and  for suggesting \Cref{cor:BAothersplit}.

\section{An involution on the intersection of two opposite maximal Schubert cells}
\label{sec:heisen}

The goal of this section is to state and prove the main technical result 
behind the results discussed in the introduction.
See \Cref{thm:swaps} below.

We recall that $\SL(d,\R)$ acts transitively on the set consisting of all antipodal pairs of points in $\F_d$.
From now on, we reserve the notation $\sigma_\pm$ for the {\em descending/ascending} flags defined as follows: Let $\R^d$ be equipped with the standard basis $\{e_1,\dots,e_d\}$. Define
\begin{align*}
 \sigma_+ &: \quad\{0\} \subset  {\rm span}\{e_d\} \subset {\rm span}\{e_d,e_{d-1}\} \subset \dots\subset {\rm span}\{e_d,\dots,e_1\} = \R^d,\\
 \sigma_- &: \quad\{0\} \subset  {\rm span}\{e_1\} \subset {\rm span}\{e_1,e_{2}\} \subset \dots\subset {\rm span}\{e_1,\dots,e_d\} = \R^d.
\end{align*}
It can be seen easily that $\sigma_\pm$ are antipodal.

We also reserve the notation $U_d$ to denote the subgroup of $\SL(d,\R)$ consisting of all upper-triangular unipotent matrices.
It is easy to check that $U_d$ fixes $\sigma_-$, and hence, preserves the big cell $\C_{\sigma_-}$. Moreover, $U_d$ acts on $\C_{\sigma_-}$ simply transitively, so we have a diffeomorphism
\[
 F_{\sigma_+} : U_d \to \C_{\sigma_-},\quad
 F(u) = u\sigma_+.
\]
For notational convenience, for all $\sigma\in \C_{\sigma_-}$, let us denote 
\[u_\sigma \coloneqq F_{\sigma_+}^{-1}(\sigma).\]
We identify $U_d$ with $\C_{\sigma_-}$ under the diffeomorphism $F_{\sigma_+}$.

Furthermore, we identify $U_d$ (hence $\C_{\sigma_-}$) with $\R^{d\choose 2}$ by sending a matrix $u\in U_d$ to the vector $(u_{ij})_{1 \le i<j\le d}$.
Under this identification,  $\sigma\in \C_{\sigma_-}$ lies in $\E_{\sigma_+} = \F_d\setminus \C_{\sigma_+}$ if and only if there exists some $k\in \{1,\dots,d-1\}$ such that $\sigma^{(k)} + \sigma_+^{(d-k)}$ is a proper subspace of $\R^d$ or, equivalently,
\[
p_k(u_\sigma) \coloneqq
  \frac{(u_{\sigma} e_{d-k+1}) \wedge \dots \wedge (u_{\sigma} e_{d}) \wedge e_{k+1} \wedge \dots \wedge e_{d}}{e_1\wedge\dots \wedge e_d } = 0.
\]
Thus, we can describe the set $\E_{\sigma_+}\cap \C_{\sigma_-}$ algebraically as a subset of $\R^{d\choose 2}$ by 
\[
 \E_{\sigma_+}\cap U_d = \bigcup_{k=1}^{d-1}  \E^k_{\sigma_+},
\]
where $\E^k_{\sigma_+} \coloneqq \{u\in U_d \mid p_k(u) = 0 \}$.

\begin{example}
 If $d = 3$, then
 \[
  p_1\left( \begin{bmatrix}
1 & x & y\\
 & 1 & z\\
 & & 1
\end{bmatrix} \right) = y,
\quad
p_2\left( \begin{bmatrix}
1 & x & y\\
 & 1 & z\\
 & & 1
\end{bmatrix} \right) = xz - y.
 \]
 Therefore,
 $\E_{\sigma_+}\cap U_3$ can be written as the union of the hypersurfaces $\E^1_{\sigma_+}=\{(x,y,z) \mid p_1(x,y,z) = y =0 \}$ and $\E^2_{\sigma_+} = \{ (x,y,z) \mid p_2(x,y,z) = xz - y = 0\}$ in $\R^3$.
 See \Cref{fig:flag_3}.
\end{example}

\begin{figure}[h]
  \centering
  \includegraphics[scale=.4]{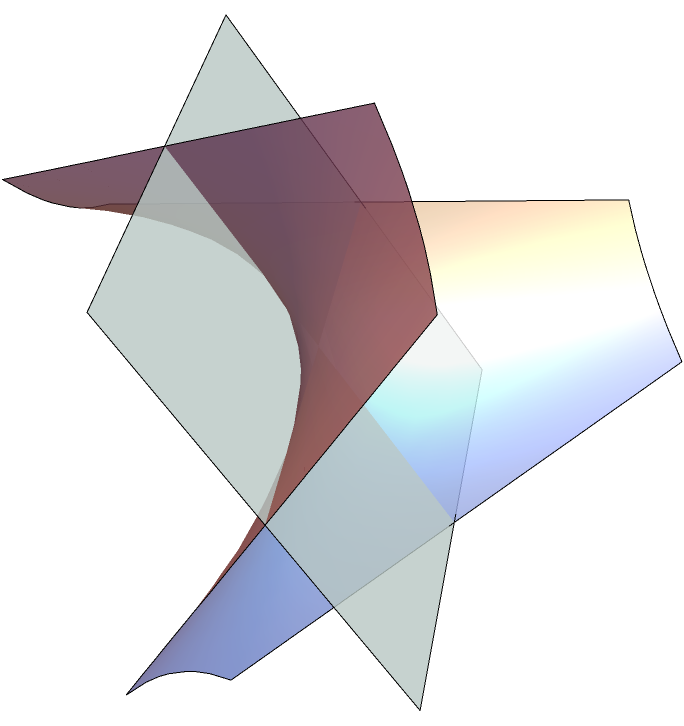}
  \caption{The part of the set $\E_{\sigma_+}$ lying in $\R^3 \cong \C_{\sigma_-} \subset \F_3$.
  The six components of $\C_{\sigma_-}\cap \C_{\sigma_+}$ are visible in the complement of this algebraic surface.}
  \label{fig:flag_3}
\end{figure}

The following lemma can be verified by linear algebra. We omit the details.

\begin{lemma}
 The polynomial $p_k(u)$ can be expressed as
 $
  p_{k}(u) = \det u^{(k)},
 $
 where  $u^{(k)}$ denotes the upper-right $k\times k$ block sub-matrix of $u$.
 In particular,
 \[
  \E^k_{\sigma_+} = \{ u\in U_d \mid \det u^{(k)} = 0  \}.
 \]
\end{lemma}

We define an involution 
\[
 \iota : U_d \to U_d, \quad  u \mapsto u^{-1}. 
\]
This simple involution plays a key role in this paper.
Note that $\iota$ is a diffeomorphism, ${\rm Fix}(\iota) = \{{\mathbf{I}}\}$, where ${\mathbf{I}}$ denotes the identity matrix, and $d\iota\vert_{T_{\mathbf{I}} U_d} = -{\rm id}$.

\begin{proposition}\label{prop:jac}
For all $k\in \{1,\dots,d-1\}$,
 \[p_k(u^{-1}) = (-1)^{k(d+1)}p_{d-k}(u).\] 
 In particular, $\iota(\E_{\sigma_+}^k) = \E_{\sigma_+}^{d-k}$, $\iota$ preserves $\E_{\sigma_+}\cap U_d$ and, hence, preserves $U_d\setminus \E_{\sigma_+}$.
\end{proposition}

\begin{proof}
 We apply the {\em Jacobi's complementary minor formula}: If $A$ is an invertible $d\times d$ matrix, then for any subsets $I,J \subset \{1,\dots,d\}$ of size $k$,
$$
  \det A_{IJ} = (-1)^{\sum I + \sum J} (\det A ) \det \left((A^{-1})_{J^c  I^c }\right).
$$
 Here we use the notation $A_{IJ}$ to denote the sub-matrix of $A$ obtained by its $I$'th rows and $J$'th columns.
 Since in our case $\det u = 1$, the above formula reduces to 
 \[\det ((u^{-1})_{IJ}) = (-1)^{{\sum I + \sum J}} \det u_{J^c I^c}.\]
 
 Fix $k\in \{1,\dots, d-1\}$.
 Notice that when $I = \{1,\dots, k\}$ and $J = \{d-k+1,\dots, d\}$, we have $p_k(\hat u) = \det \hat u_{IJ}$, $p_{d-k} (\hat u) = \det \hat u_{J^c I^c}$, and $\sum I + \sum J = k(d+1)$.
 Hence, by the formula in the previous paragraph,
 $
 p_{k}( u^{-1}) = (-1)^{k(d+1)} p_{d-k}(u).
 $
\end{proof}

By the above result, we thus have a well-defined involution $\iota$ on $U_d \setminus \E_{\sigma_+}$.
Our main result of this section is as follows:

\begin{theorem}\label{thm:swaps}
Suppose that $d$ is any natural number such that $d\ne 5$ and $d \equiv 2,3,4,5$, or $6 \mod 8$.
 Then, the involution $\iota : U_d \to U_d$ does not leave invariant any connected components of $U_d \setminus \E_{\sigma_+}$.
\end{theorem}

The proof of \Cref{thm:swaps} is split into several cases and occupies the rest of this section. Here is our plan:
The proof for the case $d=3$ is discussed in \S\ref{case:d3} and see \S\ref{case:d2mod4} for the case when $d$ is of the form $4k+2$.
These initial cases are approached in an elementary manner.
However, as our elementary approach appears to be insufficient for the remaining cases, we rely upon some sophisticated invariants developed in Shapiro--Shapiro--Vainshtein \cite{MR1446839,MR1631773}, which characterizes the connected components of $\C_{\sigma_-}\cap \C_{\sigma_+}$ in a combinatorial manner. In \S\ref{sec:SSV}, we recall some necessary background on these papers. Subsequently, the proof for $d=4$ is discussed in \S\ref{sec:d4}. Following that, we prove the theorem in the rest of the odd cases of $d$ in \S\ref{sec:dodd} and in the remaining even cases of $d$ in \S\ref{sec:deven}.

\begin{remark}\label{rem:swaps}
 When $d$ is of the form $8m\pm 1$, then the $\iota : U_d \to U_d$ leaves invariant some components of $U_d \setminus \E_{\sigma_+}$, see below.
 Therefore, \Cref{thm:swaps} is false in those cases of $d$. Cf. \Cref{proof:oddcases}(ii).
 When $d$ is of the form $8m$, we are unable to make the conclusion because we could not study some ``exceptional'' connected components.
 See \Cref{rem:8m}.
 However, in all these cases, the total number of these components is quite ``small'' compared to the total number of connected components of $U_d \setminus \E_{\sigma_+}$.

However, with help from Su Ji Hong, we managed to computationally verify \Cref{thm:swaps} for $d=5$.  Despite this effort, we decided to exclude the $d=5$ case because we couldn't find a way to present a proof for it.
\end{remark}

\subsection{Proof of Theorem \ref*{thm:swaps} when $d=3$}\label{case:d3}

When $d=3$, the connected components of $U_3 \setminus \E_{\sigma_+}$ in $\R^3 = U_3$ are
\begin{align*}
   \Omega_1 &= \{ (x,y,z) \mid x>0,~y>0,~z >0, ~xz - y >0 \},\\
   \widehat\Omega_1 &= \{ (x,y,z) \mid x<0,~y>0,~z<0, ~xz - y >0\},\\
   \Omega_2 &= \{ (x,y,z) \mid x<0,~y<0,~z >0, ~xz - y <0 \},\\
   \widehat\Omega_2 &= \{ (x,y,z) \mid x>0,~y<0,~z<0, ~xz - y <0 \},\\
   \Omega_3 &= \{ (x,y,z) \mid y>0, ~xz - y <0 \},\\
   \widehat\Omega_3 &= \{ (x,y,z) \mid y<0, ~xz - y >0 \}.
\end{align*}
See \Cref{fig:flag_3}.
By picking a representative in each component and applying $\iota$ to the representative, it can be checked that, for $k=1,2,3$, $\iota\Omega_k = \widehat\Omega_k$.
We omit the details.

\subsection{Proof of Theorem \ref*{thm:swaps} when $d \equiv 2 \mod 4$}\label{case:d2mod4}

Suppose that $d \equiv 2 \mod 4$.
Let $u\in U_d \setminus \E_{\sigma_+}$ be any point.
 By \Cref{prop:jac}, $\iota u\in U_d \setminus \E_{\sigma_+}$.
 Let $c: [-1,1] \to U_d$, $c(\pm1) = u^{\pm 1}$, be any path.
 We show that such a path $c$ must intersect $\E_{\sigma_+}\cap U_d$.
 In this case, since $d/2$ is odd, by \Cref{prop:jac}, 
 \[p_{d/2}(u^{-1}) = -p_{d/2}(u).\]
 Thus, by continuity, the image of $c$ must intersect $\E^{d/2}_{\sigma_+}$.
 Therefore, $u$ and $\iota u$ lie in different connected components of $U_d \setminus \E^{d/2}_{\sigma_+}$, and hence, of $U_d \setminus \E_{\sigma_+}$.

\subsection{Some preparation before the proof of Theorem \ref*{thm:swaps} in the remaining cases}\label{sec:SSV}

Before discussing the proof of \Cref{thm:swaps} in the rest of the cases, we recall some notions from Shapiro--Shapiro--Vainshtein \cite{MR1446839,MR1631773}.
Throughout, we try to be consistent with their papers so that we can freely refer to those  for more details.

Let $n \coloneqq d-1$.
Denote by $T^n = T^n(\FF_2)$ the vector space of all $n\times n$ upper-triangular matrices with $\FF_2$-valued entries, where $\FF_2 = \{0,1\}$ is the finite field of order $2$.
There is certain subgroup $\GG_n < {\rm GL}(T^n)$ acting linearly on $T^n$.
This action called the {\em first} $\GG_n$-action. See the introduction of \cite{MR1631773}.
There is also another $\GG_n$-action defined in that paper, which is called the {\em second} $\GG_n$-action. However, in the present article, we do not need to discuss the second action, and hence, we will call the {\em first $\GG_n$-action} simply by {\em $\GG_n$-action}.
For reader's convenience, we recall this action.
For $1\le i \le j \le n-1$, let $g_{ij} \in {\rm GL}(T^n)$ be the element acting linearly on $T^n$ as follows:
Let $M^{ij}$ be the $2\times 2$ sub-matrix of $M$ formed by the rows $i$ and $i+1$, and the columns $j$ and $j+1$ (or, its upper-triangle when $i=j$). Then, $g_{ij} \cdot M$ is the matrix obtained by adding to each entry of $M^{ij}$ its trace, and keeping the rest of the entries of $M$ unchanged.
The subgroup $\GG_n < {\rm GL}(T^n)$ is generated by all these $g_{ij}$'s.

\begin{example}
For an illustration, we revisit the case $d = n+1=3$: In this case,
 \begin{equation}\label{eqn:T2}
 T^2 = \left\{
 \begin{bmatrix}0 & 0\\ & 0 \end{bmatrix},
 \begin{bmatrix}0 & 1\\ & 0 \end{bmatrix},
 \begin{bmatrix}1 & 0\\ & 1 \end{bmatrix},
 \begin{bmatrix}1 & 1\\ & 1 \end{bmatrix},
 \begin{bmatrix}0 & 0\\ & 1 \end{bmatrix},
 \begin{bmatrix}1 & 1\\ & 0 \end{bmatrix},
 \begin{bmatrix}1 & 0\\ & 0 \end{bmatrix},
 \begin{bmatrix}0 & 1\\ & 1 \end{bmatrix}
 \right\},
\end{equation}
and $\GG_2$ is the group of order 2, with $g_{11}$ as its nontrivial generating element.
We note that $\GG_2$ precisely has four fixed points in $T^2$, represented by the initial four elements in \eqref{eqn:T2}.
The remaining four elements in $T^2$ form two distinct $\GG_2$-orbits.
Consequently, there are six $\GG_2$-orbits in total.
Remarkably, each of these orbits corresponds to a unique connected component of $U_3 \setminus \E_{\sigma_+}$ (see \S\ref{case:d3}).
We elaborate this further in our discussion below.

When $d= n+1=4$, the $\GG_3$-orbits in $T^3$ can be found in \Cref{table}.
\end{example}

By \cite{MR1446839}, the connected components of $\C_{\sigma_-}\cap \C_{\sigma_+}$ are in one-to-one correspondence with the $\GG_n$-orbits in $T^n$.
The correspondence can be realized as follows (see \cite[\S 2, \S 3]{MR1446839} for more details): Let $S_{n+1}$ denote the group of all permutations of $\{1,\dots,n+1\}$, let $w_0$ denote the {\em longest element} in $S_{n+1}$, and
let $s_k$, $k=1\dots,n$, denote the transposition which permutes $k$ and $k+1$ in $\{1,\dots,n+1\}$.
The element $w_0$ can be written as
\begin{equation}\label{eqn:order}
  w_0 = (s_1s_2\dots s_{n})( s_1 s_2 \dots s_{n-1}) \dots (s_1 s_2 s_3)(s_1 s_2) (s_1).
\end{equation}
Corresponding  to this (fixed) reduced decomposition of $w_0$, by \cite{MR1405449,MR1327548}, a generic matrix $u\in U_d = U_{n+1}$ can be uniquely factorized as
\begin{equation}\label{eqn:factorization}
\begin{split}
   u = ({\mathbf{I}}+ t_{1n} {\mathbf E}_{s_{1}})({\mathbf{I}}&+ t_{2(n-1)} {\mathbf E}_{s_2})\dots ({\mathbf{I}}+ t_{n1} {\mathbf E}_{s_{n}})\\
  &({\mathbf{I}}+ t_{1(n-1)} {\mathbf E}_{s_1})\dots  ({\mathbf{I}}+ t_{(n-1)1}{\mathbf E}_{s_{n-1}})\dots\\
  &\hspace{5em}\dots ({\mathbf{I}}+ t_{12} {\mathbf E}_{s_{1}}) ({\mathbf{I}}+ t_{21} {\mathbf E}_{s_{2}})
  ({\mathbf{I}}+ t_{11} {\mathbf E}_{s_{1}}),
\end{split}
\end{equation}
where $t_{ij}$ represents the coefficient of ${\mathbf E}_{s_i}$ when ${\mathbf E}_{s_i}$ appears the $j$-th time from the right to left in the above expression,
 $t_{ij}$ are non-zero real numbers, and
${\mathbf E}_{s_i}$ denotes the $(n+1)\times (n+1)$ matrix with only non-zero entry $1$ at the place $(i,i+1)$.
Using this unique factorization of $u$, we assign to it the matrix $M_u\in T^n$ given by 
\begin{equation}\label{eqn:Mu}
M_u =
  \begin{bmatrix}
\epsilon_{11} & \epsilon_{21} & \dots & \epsilon_{(n-1)1} & \epsilon_{n1}  \\
 & \epsilon_{12} &  & & \epsilon_{(n-1)2}  \\
 &  & \multicolumn{2}{c}{\smash{\raisebox{.5\normalbaselineskip}{\diagdots{5em}{.5em}}}} & \vdots \\ 
 &  & \multicolumn{2}{c}{\smash{\raisebox{.5\normalbaselineskip}{\diagdots{5em}{.5em}}}}   & \epsilon_{2(n-1)} \\
 &  &  & & \epsilon_{1n} \\
\end{bmatrix},
\end{equation}
where $\epsilon_{ij} = 0$ if $t_{ij}>0$ in \eqref{eqn:factorization}, and $\epsilon_{ij} = 1$ if $t_{ij}<0$. Cf. \cite[\S2.8]{MR1446839}.

\begin{lemma}\label{lemma:inversionformula}
 For a generic matrix $u\in U_d$, 
 the matrices $M_u, M_{u^{-1}}$ are related by
 \[
  (M_{u^{-1}})_{ij} = (M_u)_{(n+1-j) (n+1-i)} +1, \quad \text{for all } 1\le i\le j\le n.
 \]
\end{lemma}

\begin{proof}
Suppose that the factorization of $u$ is given by \eqref{eqn:factorization} and the corresponding $M_u$ has the expression given by \eqref{eqn:Mu}.
To obtain an expression for $M_{u^{-1}}$, we first notice
\begin{multline*}
  u^{-1} = ({\mathbf{I}}- t_{11} {\mathbf E}_{s_{1}})({\mathbf{I}}- t_{21} {\mathbf E}_{s_{2}})({\mathbf{I}}- t_{12} {\mathbf E}_{s_{1}}) \dots\\
  \dots ({\mathbf{I}}- t_{(n-1)1}{\mathbf E}_{s_{n-1}}) \dots ({\mathbf{I}}- t_{1(n-1)} {\mathbf E}_{s_1})\\
  ({\mathbf{I}}- t_{n1} {\mathbf E}_{s_{n}})\dots
  ({\mathbf{I}}- t_{2(n-1)} {\mathbf E}_{s_2})({\mathbf{I}}- t_{1n} {\mathbf E}_{s_{1}}) 
\end{multline*}
Observe that this expression is not in the correct order required by the chosen reduced form of $w_0$ in \eqref{eqn:order}; cf. \eqref{eqn:factorization}.
However, using the fact that ${\mathbf E}_{s_{i}}$ and ${\mathbf E}_{s_{j}}$ commute if $|i-j|\ge 2$, we can put this expression easily in the desired form \eqref{eqn:factorization}:
If $t_{ij}'$ denote the coefficients involved in this expression for $u^{-1}$, then 
\[t_{ij}' = -t_{i(n+1-j)}.\]
Hence, the matrix $M_{u^{-1}}$ is derived from $M_u$ through a two-step process: first, reflecting it across its anti-diagonal to adjust the indices, and then adding $1$ to each entry in the upper-triangular region to accommodate the sign changes.
This verifies the lemma.
\end{proof}

To each connected component $\Omega$ of $\C_{\sigma_-}\cap \C_{\sigma_+}$, we associate the set
\[S_\Omega \coloneqq \{M_u \mid u\in\Omega \text{ is generic}\} \subset T^n.\]
The vector space $T^n$ partitions into the subsets of the form $S_\Omega$ and the correspondence $\Omega \leftrightarrow S_\Omega$ is one-to-one.
Moreover, 
by the Main Theorem of \cite{MR1446839}, the subsets $S_\Omega$ are precisely the orbits of the $\GG_n$-action.
Let us define an involution $\iota : T^n \to T^n$ by
\begin{equation}\label{eqn:ipdef}
  \iota(M)_{ij} \coloneqq M_{(n+1-j) (n+1-i)} +1, \quad \text{for all } 1\le i\le j\le n.
\end{equation}
A consequence of \Cref{lemma:inversionformula} is that
\begin{equation}\label{eqn:ip}
  \iota S_\Omega = S_{\iota\Omega}.
\end{equation}
Thus, $\Omega$ is $\iota$-invariant if and only if $S_\Omega$ is.
The first part of the following lemma records this discussion.

\begin{lemma}\label{lemma:io}
 Let $n\in\N$. The involution $\iota : U_{n+1} \to U_{n+1}$ preserves a connected component $\Omega$ of $\C_{\sigma_-}\cap \C_{\sigma_+}$ if and only if the map $\iota : T^n \to T^n$ leaves $S_\Omega$ invariant.
 
 Moreover, the involution $\iota$ has no fixed points in $T^n$.
 In particular, no singleton $\GG_n$-orbits are preserved by $\iota$. 
\end{lemma}

The ``moreover'' part of the lemma above is verified by noticing that the entries in upper-right corner of $M \in T^n$ and $\iota(M)$ are  different.

\medskip
We identify the dual space $(T^n)^*$ with the space of $n\times n$ upper-triangular matrices with $\FF_2$-entries so that, for $M\in T^n$ and $M^*\in (T^n)^*$, 
\begin{equation}\label{eqn:pairing}
 \langle M,M^*\rangle = \sum_{i\le j} M_{ij} M^*_{ij}.
\end{equation}
We recall the elements $E_k\in T^n$ and $R_k \in (T^n)^*$, $k= 1,\dots, n$, from \cite[\S 2.1]{MR1631773}
\[
 E_k = \sum_{s-r = k-1} E_{rs},\quad \text{and} \quad
 R_k = \sum_{1\le r\le k\le s\le n} E_{rs},
\]
where $E_{rs}$ denotes the matrix whose only nontrivial entry is at the position $(r,s)$.
The subspace of $(T^n)^*$ (resp. $T^n$) spanned by the matrices $R_k$'s (resp. $E_k$'s) is denoted by $\mathcal{D}_n$ (resp. $\mathcal{I}_n$).
One checks that $\iota E_{k} = \sum_{i\ne k} E_i$, and hence,
\begin{equation}\label{eqn:iotain}
  \iota \mathcal{I}_n = \mathcal{I}_n. 
\end{equation}
Moreover, note that the matrices $E_k$ are symmetric with respect to the anti-diagonal; therefore, any element $I\in\mathcal{I}_n = {\rm span}\{E_1,\dots,E_n\}$ is also symmetric with respect to the anti-diagonal.
Hence,
\begin{equation}\label{eqn:iota_commutes}
  \iota(I + M) = I +\iota(M), \quad
 \forall M\in T^n,\, \forall I\in \mathcal{I}_n.
\end{equation}

Let $\mathcal{D}_n^\perp \subset T^n$ denote the subspace orthogonal to $\mathcal{D}_n$ with respect to the standard pairing $\langle \cdot,\cdot\rangle$ in \eqref{eqn:pairing}.
A translation of $\mathcal{D}_n^\perp$ by a matrix $M\in T^n$ is called a {\em slice}.
If $S \subset T^n$ is a slice, then its {\em height} $h^S$ is defined to be the vector \[h^S = (h_1^S,\dots,h_n^S) \in \FF^n_2,\] where $h_k^S \coloneqq \langle M,R_k\rangle \in\FF_2$ and  $M\in S$ is a(ny) matrix.
A straightforward computation using \eqref{eqn:ipdef} shows that for any $M\in T^n$,
\[
 \langle \iota (M), R_{n+1-k} \rangle \equiv
 \langle M, R_k \rangle + k(n+1-k)\mod 2.
\]
In particular, we see that $\iota$ sends slices to slices.
The following lemma relates the height vectors of $S$ and ${\iota S}$, which follows from the formula above:

\begin{lemma}\label{lem:height}
 For every slice $S \subset T^n$ and $k\in\{1,\dots,n\}$, 
 \[h_k^S \equiv h^{\iota S}_{n+1-k} + k(n+1-k)\mod 2.\]
\end{lemma}

Note that $\mathcal{D}_n^\perp$ is the slice at height zero.
Moreover, the correspondence $S \leftrightarrow h^S$ is one-to-one.
A slice $S$ is called {\em symmetric} if its height vector $h^S$ is a symmetric vector with respect to its middle, i.e., for all $k\in\{1,\dots,n\}$, $h_k^S = h^S_{n+1-k}$.
By the $\GG_n$-orbit structure theorem \cite[Theorem 2.2]{MR1631773}, every orbit of $\GG_n\acts T^n$ lies in some slice $S\subset T^n$.

With the help of \Cref{lemma:io} and \Cref{lem:height}, our strategy now is to apply \cite[Theorem 2.2]{MR1631773} to check if any $\GG_n$-orbit is preserved under the involution $\iota$.

\subsection{Proof of Theorem \ref*{thm:swaps} when $d=4$}\label{sec:d4}

We give a proof of \Cref{thm:swaps} when $d=4$. This case is illustrative, and also does not fit into the discussion of the more general cases below.
In this case, $T^3$ has $64$ elements, and there are twenty $\GG_3$-orbits, each corresponds to one connected component of $U_4 \setminus \E_{\sigma_+}$.
The orbits are listed in \Cref{table}.
The theorem can be verified directly from the table.

\subsection{Proof of Theorem \ref*{thm:swaps} when $d$ is odd and $d\ge 7$}\label{sec:dodd}
More precisely, we prove the following:

\begin{proposition}\label{proof:oddcases}
 Let $d\ge 7$ be an odd integer.
\begin{enumerate}[label=(\roman*)]
 \item If $d\equiv 3$ or $5\mod 8$, then the involution $\iota : U_d \to U_d$ does not preserve any connected component of $U_d \setminus \E_{\sigma_+}$.
 \item If $d\equiv 1$ or $7\mod 8$, then $\iota$ preserves $2^{\frac{d+1}{2}}$ connected components of $U_d \setminus \E_{\sigma_+}$.
\end{enumerate}
\end{proposition}

Note that there are $3\cdot 2^{d-1}$ connected components of $U_d \setminus \E_{\sigma_+}$ \cite{MR1631773}.
 
Let $d\ge 7$ be any odd integer.
Equivalently, we assume that $n = d-1 \ge 6$ is even.
Applying \Cref{lem:height}, we see that for all $k\in\{1,\dots,n\}$, $h_k^S = h^{\iota S}_{n+1-k}$.
Thus, if $S \subset T^n$ is a non-symmetric slice, then $h^{\iota S} \ne h^{S}$.
Hence, $\iota$ does not preserve any orbits lying in the non-symmetric slices.

However, for every symmetric slice $S$, 
\begin{equation}\label{eqn:symheight}
 h^S = h^{\iota S}
\end{equation}
or, equivalently,  $\iota S = S$.
Using \cite[Theorem 2.2(ii)]{MR1631773}, every symmetric slice $S$ decomposes into a number of singleton $\GG_n$-orbits and two non-singleton $\GG_n$-orbits of equal sizes.
By \Cref{lemma:io}, no singleton $\GG_n$-orbit is preserved under $\iota$.
So, it remains only to check how $\iota$ acts on the pair of non-singleton $\GG_n$-orbits in each symmetric slices.

Any symmetric slice can be sent to any other by the action $\mathcal{I}_n \acts T^n$ by translations. 
It has been noted in the proof of \cite[Lemma 6.7]{MR1631773} that $\mathcal{I}_n \acts T^n$ maps $\GG_n$-orbits to $\GG_n$-orbits.

\begin{claimone} 
 If $\iota$ swaps (resp. preserves) the pair of non-singleton $\GG_n$-orbits in one symmetric slice, then it swaps (resp. preserves) those for all symmetric slices.
\end{claimone}

\begin{proof}
 Let $S_1$ and $S_2$ be any two symmetric slices, and let $I \in \mathcal{I}_n$ be a matrix such that $I + S_1 = S_2$.
 Let $S^\pm_1 \subset S_1$ denote the distinct non-singleton $\GG_n$-orbits.
 Then, $S^\pm_2 \coloneqq I+ S^\pm_1 \subset S_2$ are the distinct non-singleton $\GG_n$-orbits in $S_2$.
 If $\iota (S^+_1) = S_1^-$, then, by \eqref{eqn:iota_commutes},
 \[
  \iota(S^+_2) = \iota (I + S^+_1) = I + \iota (S^+_1) = I + S_1^- = S_2^-.\qedhere
 \]
\end{proof}

Thus, it is enough to understand how $\iota$ acts on the pair of non-singleton $\GG_n$-orbits in $\mathcal{D}_n^\perp$, the symmetric slice at zero height.
Consider the matrix $M^-_n\in T^n$ whose only nontrivial entries are the ones contained in the $2\times 2$ sub-matrix at the upper-right corner, and let $M^{+}_n \coloneqq \iota (M^{-}_n)$. 
We note that, $M^-_n\in \mathcal{D}^\perp_n$ and hence, so is $M_n^+$.  Using the description of the $\GG_n$-action above, it is easy to observe that the $\GG_n$-orbits $S^\pm_n \coloneqq  \GG_n \cdot M^\pm_n$ are both non-singleton.
Finally, since $M^{+}_n = \iota (M^{-}_n)$, we get $S_n^+ = \iota(S_n^-)$.

\begin{claimtwo}
For all even numbers $n\ge 6$, $S^+_n \cap  S^-_n =\emptyset$ precisely when $n\equiv 2$ or $4 \mod 8$.
\end{claimtwo}

\begin{proof}
 Let $\Phi_n : (T^n)^* \to T^{n-1}$ denote the linear map given by sending a matrix $M\in (T^n)^*$ to $N\in T^{n-1}$ such that
 \[
  N_{ij} = M_{ij} + M_{i+1,\,j} + M_{i,\,j+1} + M_{i+1,\,j+1}.
 \]
 It is proven in \cite[Lemma 6.6]{MR1631773} that the dual map
 $\Phi_n^* : (T^{n-1})^* \to T^n$ maps $(T^{n-1})^*$ isomorphically onto $\mathcal{D}^\perp_n$.
 The dual map $\Phi^*$ can be computed by
\begin{equation}\label{eqn:dualphi}
 \Phi^*(E^{n-1}_{ij}) = E^n_{ij} + E^n_{i,\, j+1} + E^n_{i+1,\,j} + E^n_{i+1,\,j+1},
\end{equation}
 where $E^k_{ij}$ denotes the $k\times k$ matrix with only nontrivial entry at the position $(i,j)$, if $i \le j$, or the zero $k\times k$ matrix, otherwise.
 Let $N^-_{n-1} \in (T^{n-1})^*$ be the matrix whose only nontrivial entry is contained in the upper-right corner.
 Let $N^+_{n-1} = N^-_{n-1} + P_{n-1}$, where $P_{n-1}$ denotes the $(n-1)\times(n-1)$ matrix whose nontrivial entries are precisely located at the places $(i,j)$ such that $i\le j$, and $i$ and $j$ are both odd numbers.
 By the description of $\Phi^*$ above, it is easy to check that
  \[\Phi_n^*(N^\pm_{n-1}) = M^\pm_n.\]
 There is a quadratic function 
\begin{equation}\label{eqn:quad}
  Q : (T^{n-1})^* \to \FF_2
\end{equation}
 defined in \cite[\S 5.1]{MR1631773} which distinguishes between the pair of non-singleton orbits. 
 Applying \cite[Lemma 5.1]{MR1631773}, we get
 \[
  Q(N^-_{n-1}) = 1 \quad \text{and} \quad Q(N^{+}_{n-1}) = \frac{\frac{n}{2}(\frac{n}{2} +1)}{2} -1 \mod 2.
 \]
 Note that the quantity $\frac{\frac{n}{2}(\frac{n}{2} +1)}{2}$ counts the number of $1$'s in the matrix $P_{n-1}$.
 Therefore, $Q(N^+_{n-1}) \ne Q(N^{-}_{n-1})$ exactly in the cases when $n\equiv 2$ or $4 \mod 8$.
 Applying \cite[Lemmata 4.3, 5.5 \& 6.6]{MR1631773}, the claim follows.
\end{proof}

Now we are in a position to give:

\begin{proof}[Proof of Proposition \ref*{proof:oddcases}]
Let $n = d-1$.

(i) By the second claim above, for all even integers $n\ge 6$ satisfying $n\equiv 2$ or $4 \mod 8$, $\iota$ swaps the pair of non-singleton $\GG_n$-orbits in $\mathcal{D}^\perp_n$. 
Following the discussion before that claim, we conclude that no orbit of $\GG_n \acts T^n$ is preserved by $\iota$.

(ii) If $n\ge 6$ and $n\equiv 0$ or $6 \mod 8$, then by the above claim it follows that $\iota S^+_n = S^-_n = S^+_n$. Hence, $\iota$ preserves the non-singleton orbits of $\GG_n\acts T^n$ lying in the symmetric slices. Finally, by the first item of \cite[Theorem 2.2(ii)]{MR1631773}, there are exactly $2^{\frac{n}{2}+1}$ such orbits.
\end{proof}

\subsection{Proof of Theorem \ref*{thm:swaps} when $d \equiv 4\mod 8$ and $d\ge 12$}\label{sec:deven}

The only remaining case of \Cref{thm:swaps} is as follows:

\begin{proposition}
  Let $d\ge 12$ be an integer such that $d \equiv 4\mod 8$. 
  The involution $\iota : U_d \to U_d$ does not preserve any connected components of $U_d \setminus \E_{\sigma_+}$.
\end{proposition}

\begin{proof}
 Suppose that $n = d-1\ge 11$ is an odd integer such that $n\equiv 3\mod 8$.
Applying \Cref{lem:height}, we observe that the $\GG_n$-orbits in the {\em symmetric} slices are not preserved, since, for every symmetric slice $S$, $h_1^S = 1+ h_1^{\iota S}$.
Furthermore, by a similar application of \Cref{lem:height} to the non-symmetric slices $S$, we observe that $h^S \ne h^{\iota S}$ unless $h^S$ satisfies 
\begin{align}\label{eqn:nonsym}
\begin{split}
 \text{for all even }k,\quad h^S_k &= h^{S}_{n+1-k}, \quad\text{ and}  \\
 \text{for all odd }k,\quad h^S_k &= h^{S}_{n+1-k} +1. 
\end{split}
\end{align}

Therefore, our discussion reduces to the case of $\GG_n$-orbits contained in the {\em non-symmetric} slices whose height vectors $h^S$ satisfy \eqref{eqn:nonsym}; we call such {non-symmetric} slices {\em special}. 
By definition, it follows that a slice $S$ is {special} if and only if $\iota(S) = S$.
By \cite[Theorem 2.2(i)]{MR1631773}, every {non-symmetric} (in particular, special) slice decomposes into a pair of orbits of equal sizes.

\begin{claim}
 Any special slice can be brought to any other by the action $\mathcal{I}_n \acts T^n$ by translations.
\end{claim}

\begin{proof}
 If $S$ and $S'$ are any two special slices, then the difference vector $h^S - h^{S'}$ is symmetric with respect to the middle.
 We only need to remark that the image of the map $h : \mathcal{I}_n \to \FF^n_2$ which sends a matrix $M\in \mathcal{I}_n$ to the vector $(h_1^M,\dots,h_n^M) \in \FF^n_2$, where $h_k^M \coloneqq (M,R_k) \in\FF_2$, 
 consists of all vectors $h\in\FF^n_2$ which are symmetric with respect to the middle.
\end{proof}

Since $\iota$ preserves the orbit structure of the action $\mathcal{I}_n \acts T^n$ (by \eqref{eqn:iota_commutes}), and the action $\mathcal{I}_n \acts T^n$ preserves the $\GG_n$-orbit structure of the slices, by the above claim, it is enough to understand the involution $\iota : S \to S$ on {\em only} one special slice $S$.
Let $\bar S_n$ denote the special slice at height
\[
 \bar h_n = (\underbrace{1,0,1,0,\dots,1}_{\text{first $\frac{n-1}{2}$ entries}},0,\dots,0) \in \FF^n_2.
\]
Note that $\bar h_n$ satisfies \eqref{eqn:nonsym}.
Let $\bar M^-_n \in \bar S_n \subset T^n$ denote the diagonal matrix whose diagonal entries are given by the vector $\bar h_n$, and let $\bar M^+_n \coloneqq \iota(\bar M^-_n)$. 
Let $\bar S^\pm_n \coloneqq \GG_n \cdot \bar M_n^\pm\subset \bar S_n$.

Define a map $f: T^n \to T^{n+1}$ by sending a matrix $M$ to the matrix $f(M)\in T^{n+1}$ obtained by appending the transpose of the vector
 \[
  (\underbrace{1,\dots,1}_{\text{$\frac{n+1}{2}$}},0,\dots,0) \in \FF^{n+1}_2
 \]
 to $M$ as the last column.
 By a direct calculation of the height, we observe that $f(\bar S_n) \subset \mathcal{D}^\perp_{n+1}$.
 Moreover, by definition of the $\GG_n$-action, it follows that $f(\bar S^-_n)$ (and similarly, $f(\bar S^+_n)$) is contained in a non-singleton $\GG_{n+1}$-orbit in $\mathcal{D}^\perp_{n+1}$.
Therefore, it is enough to show that $f(\bar S^-_n)$ and $f(\bar S^+_n)$ lie in two different $\GG_{n+1}$-orbits.
Recall that $Q\circ (\Phi^*_{n+1})^{-1}$, where $Q: (T^n)^* \to\FF_2^n$ is the quadratic function in \eqref{eqn:quad}, distinguishes between the pair of non-singleton $\GG_{n+1}$-orbits in $\mathcal{D}^\perp_{n+1}$.
Let $\bar N^\pm_n \in (T^n)^*$ denote the $n\times n$ matrices given by
\begin{align*}
  (\bar N^-_n)_{ij} &= 
 \begin{cases} 
      1 & i\le j,~ i \text{ is odd and }i \le \frac{n+1}{2} \\
      0 & \text{otherwise},
 \end{cases}\\
 (\bar N^+_n)_{ij} &=
 \begin{cases} 
      1 & i\le j,~ i,j \text{ are both odd and }i \le \frac{n+1}{2} \\
      1 & i\le j,~ i \text{ is odd, } j \text{ is even, and } i \ge \frac{n+1}{2} \\
      0 & \text{otherwise}.
 \end{cases}
\end{align*}
Using the description of the dual map $\Phi_n^*$ in \eqref{eqn:dualphi},
one checks that $f(\bar M^\pm_n) = \Phi_n^*(\bar N_n^\pm)$.
With the help of \cite[Lemma 5.1]{MR1631773}, we obtain that (modulo 2) the quantity $Q(\bar N_n^-)$ counts the number of nontrivial rows in $\bar N_n^-$ whereas $Q(\bar N_n^+)$ counts the number of $1$'s in $Q(\bar N_n^+)$; since $n$ is of the form $8m+3$, we get that
\[
 Q(\bar N_n^-) = 1, \quad
 Q(\bar N_n^+) = 0.
\]
This concludes the proof.
\end{proof}

\begin{remark}\label{rem:8m}
 Most of the discussion in the above proof applies to the case when $d$ is divisible by $8$, except that in this case, $Q(\bar N_{d-1}^\pm)$ are both zero.
\end{remark}

\section{Proof of Theorem \ref*{thm:main}}\label{sec:thm:main}
In this section, we prove \Cref{thm:main}.
We first need the following lemma.

\begin{lemma}\label{heisen}
Let $d$ be any natural number satisfying \eqref{eqn:d}, and let $\Omega$ be any connected component of $\C_{\sigma_-} \cap \C_{\sigma_+}$.
 Then, for every $\sigma\in \Omega$, $\sigma_+ \not\in u_\sigma \Omega$.
\end{lemma}

\begin{proof}
 The equivalent statement that, for every $\sigma\in\Omega$, $u_\sigma^{-1} \sigma_+ \not\in \Omega$, follows directly from \Cref{thm:swaps}.
\end{proof}

Now we prove \Cref{thm:main}.

\begin{proof}[Proof of \Cref*{thm:main}]
Suppose that $d\in \N$ is any number satisfying \eqref{eqn:d}, and let $\Omega$ be a connected component of $\F_d \setminus (\E_{\sigma_+} \cup \E_{\sigma_-})$.
Let $\sigma \in \Omega$ be any point.
 Pick a continuous path $u_t$, $t \in [0,1]$, in $U_d$ from the identity element to $u_\sigma$.
 The set $\E = \bigcup_{t\in[0,1]} u_t \E_{\sigma_+} = \bigcup_{t\in[0,1]} \E_{u_t\sigma_+}$ is a compact set not containing $\sigma_-$.
 Let $B_r(\sigma_-)$ denote the closed ball in $\F_d$ centered at $\sigma_-$ (with respect to some background metric on $\F_d$ compatible with the manifold topology) of radius $r>0$ small enough such that it does not intersect $\E$.
  We show that 
\begin{equation}\label{eqn}
   (B_{r}(\sigma_-) \cap \Omega) \subset u_\sigma \Omega.
\end{equation}
 By our choice of the radius $r$, any point $\s\in B_{r}(\sigma_-) \cap \Omega$ is antipodal to $\sigma_-$ and $u_t \sigma_+$, for all $t\in [0,1]$.
 Equivalently, for all $t\in [0,1]$, $u_t^{-1}\s$, is antipodal to $\sigma_-$ and $\sigma_+$.
 Therefore, we obtain a path $u_t^{-1}\s$, $0\le t \le 1$, from $\s$ to $u_1^{-1} \s$ which lies completely in a single connected component of $\C_{\sigma_-} \cap \C_{\sigma_+}$.
 Since, by assumption, $\s\in\Omega$, we must have $u_1^{-1} \s \in \Omega$.
 Hence, $\s\in u_1\Omega = u_\sigma \Omega$.
 
 Now we can complete the proof of the theorem.
 Let $c : [-1,1] \to \F_d$ be a continuous path such that
 \[
  c(\pm 1) = \sigma_\pm \quad \text{and} \quad
  c((-1,1)) \subset \Omega.
 \]
 Then, by \eqref{eqn}, there exists $t_0\in (-1,1)$ such that
\begin{equation}\label{eqn:in}
 c(t) \in u_\sigma\Omega,
 \quad \text{whenever } -1\le t\le t_0.
\end{equation}
 However, by \Cref{heisen}, $\sigma_+ \not\in u_\sigma\Omega$.
 Since $\sigma_+$ is antipodal to both $\sigma_-$ and $\sigma$,
 $\sigma_+ \not\in u_\sigma \bar\Omega$, where $\bar\Omega$ denotes the closure of $\Omega$ in $\F_d$.
 Therefore, there exists $t_1\in(t_0,1)$ such that
 \begin{equation}\label{eqn:nin}
 c(t) \not\in u_\sigma\Omega,
 \quad \text{whenever } t_1 \le t\le 1.
 \end{equation}
 By \eqref{eqn:in} and \eqref{eqn:nin}, $c([t_0,t_1])$
 must intersect the boundary $\partial (u_\sigma \Omega)$ of the subset $u_\sigma \Omega$ in $\F_d$.
 Note that the boundary of $u_\sigma\Omega$ is contained in $\E_{\sigma_-}\cup \E_{\sigma}$, because
 \[
  \partial (u_\sigma\Omega) = u_\sigma (\partial \Omega) \subset
  u_{\sigma} (\E_{\sigma_-} \cup \E_{\sigma_+}) = \E_{\sigma_-} \cup \E_{\sigma}.
 \]
 Furthermore,
  under our hypothesis, $c((-1,1))\cap \E_{\sigma_-} = \emptyset$,
 we must have
 \[
  c([t_0,t_1]) \cap \E_\sigma \neq \emptyset.
 \]
 This concludes the proof.
\end{proof}

\section{Proofs of Theorems \ref*{thm:main0} and  \ref*{thm:main1}}\label{proofs:mainresults}
We first prove \Cref{thm:main0}.
Proof of \Cref{thm:main1} is similar and given afterwards.

\begin{proof}[Proof of \Cref*{thm:main0}]\label{sec:mainthm}
 Suppose that $d\in \N$ is as in the hypothesis.
 By definition, since $\G$ is a $B$-boundary embedded subgroup of $\SL(d,\R)$, $\G$ is a hyperbolic group.
 Furthermore, since $\G$ is finitely-generated, appealing to the {Selberg lemma}, we know that $\G$ is virtually torsion-free.
 After passing to a subgroup of finite index, we may (and will) assume that $\G$ is torsion-free.
 Then, by the {Stallings decomposition theorem}, $\G$ is isomorphic to a free product
\begin{equation}\label{eqn:Sta}
   \G = F_k \star \G_1 \star \dots \star \G_n,
\end{equation}
 where $F_k$ is a free group of rank $k\ge 0$, and $\G_j$'s are one-ended hyperbolic groups ($n\ge 0$).
 Suppose that $\G$ is not free.
 Hence, we must have $n\ge 1$.
 We show that $k = 0$, $n = 1$, $\G_1$ is a surface group.
 
 The subgroup $\G_1$ is naturally $B$-boundary embedded in $\SL(d,\R)$:
 Let $\xi: \partial_\infty \G \to \F_d$ denote a (fixed) $\G$-equivariant antipodal embedding, and let $\xi_1$ denote the composition of the following maps,
 \[
  \partial_\infty \G_1 \hookrightarrow \partial_\infty \G \xrightarrow{\xi} \F_d.
 \]
 Then, $\xi_1 : \partial_\infty\G_1 \to \F_d$ is a $\G_1$-equivariant antipodal embedding.
 
\begin{lemma}
 If $\widehat\G$ is a one-ended hyperbolic group, then there exists a topological embedding $i : S^1 \to  \partial_\infty\widehat\G$.
\end{lemma}

\begin{proof}
  Such embedded circles can be constructed either by a direct topological argument by using the fact that boundaries of one-ended groups are locally connected and without any global cut points (by Swarup \cite{MR1412948}), or by using Bonk--Kleiner's \cite[Corollary 2]{MR2146190}.
\end{proof}

 Let $i : S^1 \to \partial_\infty \G_1$ be a topological embedding, and define
 \[
  c \coloneqq \xi_1 \circ i : S^1 \to \F_d.
 \]
 Then, $c$ is an antipodal map.
 Define
 \[
  \E_{c} \coloneqq \bigcup_{x\in S^1} \E_{c(x)} \subset \F_d.
 \]
 We show that $\partial_\infty \G_1 = i(S^1)$, i.e., $\partial_\infty\G_1$ is homeomorphic to a circle:
 Suppose to the contrary that $\partial_\infty \G_1 \supsetneq i(S^1)$.
 Consider a sequence of points $(y_n)$ in $\partial_\infty \G_1 \setminus i(S^1)$ which converges to some point $y\in i(S^1)$.
 Then, $\xi_1(y_n) \to \xi(y)$, as $n\to\infty$.
 Since $\xi_1$ is antipodal, $\xi_1(y_n) \not\in \E_c$.
 However, by \Cref{cor:maximal},
 the image of $c$ is contained in the interior of $\E_c$.
 Hence, we get a contradiction.
 
 Since $\partial_\infty \G_1$ is homeomorphic to a circle, $\G_1$ is isomorphic to a surface group due to the deep work by Tukia, Gabai, Freden, Casson, and Jungreis. See  Theorem 5.4 in the survey by Kapovich--Benakli \cite{MR1921706}.

 Finally, we show that $\G = \G_1$ in \eqref{eqn:Sta}:
 Suppose, to the contrary, that $\G \setminus \G_1$ is nonempty.
 Then, $\partial_\infty\G \setminus \partial_\infty \G_1$ is also nonempty (for example, the fixed points in $\partial_\infty\G$ of any element $\g\in \G \setminus \G_1$ lie outside $\partial_\infty\G_1$).
 Let $z\in  \partial_\infty\G \setminus \partial_\infty \G_1$ be an arbitrary point, and let $\g_1\in\G_1$ be a nontrivial element.
 Then, $(\g_1^n z)_{n\in\N}$ is a sequence in $\partial_\infty\G \setminus \partial_\infty \G_1$ accumulating in $\partial_\infty\G_1 \cong S^1$.
 By a similar argument as in the previous paragraph, we obtain a contradiction.
  \end{proof}
  
  Finally, we prove \Cref{thm:main1}.
  
\begin{proof}[Proof of \Cref*{thm:main1}]
If the image of $c: S^1\to \partial_\infty Z$ is not open, then
there exists a sequence $(z_n)$ in $\partial_\infty Z$ outside $c(S^1)$ converging to a point $z\in c(S^1)$. 
Suppose, to the contrary, that there exists a uniformly regular quasi-isometric embedding $f: Z\to X_d$, where $d$ satisfies \eqref{eqn:d}.
Since $Z$ is locally compact, by \cite[Theorem 1.2]{MR3890767}, $f$ admits a continuous extension
\[
 \bar f : \bar Z \to X_d \sqcup \F_d,
\]
where $\bar Z$ is the compactification of $Z$ by attaching the Gromov boundary $\partial_\infty Z$,
such that the restriction $\bar f\vert_{\partial_\infty Z}: \partial_\infty Z \to \F_d$ is an antipodal map.
As a consequence of continuity, the sequence $(\bar f(z_n))$ converges to $\bar f(z)$. However, due to the antipodality of the map $\bar f\vert_{\partial_\infty Z}$, $(\bar f(z_n))$ must remain antipodal to $\bar f(c(S^1))$.
This is a contradiction since \Cref{cor:maximal} asserts that $\bar f(c(S^1))$ is locally maximally antipodal.
\end{proof}

\section{Some further remarks}\label{sec:further}
Suppose that $d$ is any natural number satisfying \eqref{eqn:d}.
Recall that, by \Cref{cor:maximal}, antipodal circles $\Lambda$ in $\F_d$ are locally maximally antipodal.
The following result shows that, if such a circle $\Lambda$ is the limit set of some Borel Anosov subgroup of $\SL(d,\R)$, then $\Lambda$ is maximally antipodal; i.e.,
\begin{equation*}
   \bigcup_{\sigma\in\Lambda} \E_\sigma = \F_d.
\end{equation*}

\begin{proposition}\label{prop:remone}
Let $d$ be any natural number satisfying \eqref{eqn:d}.
 If $\G< \SL(d,\R)$ is a Borel Anosov subgroup, which is isomorphic to a surface group, then its {flag limit set} $\Lambda$ is a maximally antipodal subset of $\F_d$.
\end{proposition}

\begin{proof}
 Suppose, to the contrary, that there exists a point $\hat\sigma\in\F_d$ antipodal to every point in $\Lambda$.
Let $\g\in\G$ be any hyperbolic element with {\em attracting/repelling} points $\sigma_\pm \in \Lambda$. 
Then, $\gamma^k \hat\sigma \to \sigma_+$, as $k\to \infty$.
However, since $\gamma$ preserves $\Lambda$, $\gamma^k \hat\sigma$ remains antipodal to $\Lambda$, for all $k\in\N$.
Since $\Lambda$ is homeomorphic to a circle, by \Cref{cor:maximal}, $\Lambda$ is locally maximally antipodal in $\F_d$, and so we get a contradiction with the preceding two sentences. 
\end{proof}

We prove the following statement, answering a question asked by Hee Oh, which was motivated by Theorem 5.2 in Oh--Edwards \cite{Edwards:2022vr}, where the authors mention knowing the result for $d=3$ or when $d$ is even (see Remark 5.4(4) in that paper).

\begin{proposition}\label{prop:oh}
Let $d\ge 2$ be any natural number.
 The image in $\F_d$ of the equivariant limit maps corresponding to the Hitchin representations of surface groups into $\PSL(d,\R)$, $d\ge 2$, are maximally antipodal subsets.
\end{proposition}

\begin{proof}
 By \Cref{prop:remone}, this result is true for all $d$ covered under the hypothesis of \Cref{prop:remone}.
 
In any case, for all $d\ge 2$, it is enough to verify that $\Lambda$ is locally maximally antipodal in $\F_d$ (cf. the proof of  \Cref{prop:remone}):
By Fock--Goncharov \cite{MR2233852}, the Hitchin representations are characterized by $\G$-equivariant, {\em positive} limit maps $\xi: \partial_\infty \G \to \F_d$.
Let $x_-,x,x_+ \in \partial_\infty\G$ be any distinct points, and let $\sigma_\pm \coloneqq \xi(x_\pm)$ and $\sigma \coloneqq \xi(x)$.
Then, the configuration of flags $(\sigma_-,\sigma,\sigma_+)$ in $\F_d$ is {\em positive}, i.e., with an appropriate identification of $U_d$ with the unipotent radical in the stabilizer of $\sigma_-$ in ${\rm PSL}(d,\R)$, there exists a {\em totally positive} matrix $u\in U_d$ such that $\sigma = u \sigma_+$.
Such a matrix $u$ corresponds to the zero matrix $\mathbf{0} \in T^{d-1}(\FF_2)$, cf. \eqref{eqn:factorization} and \eqref{eqn:Mu}.
By \Cref{lemma:io}, the involution $\iota$ does not preserve the connected component $\Omega^+_d$ of $\C_{\sigma_-}\cap \C_{\sigma_+}$ corresponding to $\mathbf{0}$, since $\mathbf{0}$ is a $\GG_{d-1}$-fixed point for the action $\GG_{d-1} \acts T^{d-1}$; see \S\ref{sec:heisen} for these notions.
Therefore, for all $d\ge 2$, \Cref{heisen}, and hence \Cref{thm:main}, holds for the specific component $\Omega^+_d$.
Following the proof of \Cref{cor:maximal}, one verifies that $\Lambda$ is a locally maximally antipodal subset of $\F_d$.
\end{proof}

\begin{table}
\footnotesize
\centering
\begin{tabular}{ | c | c | } 
\hline
{\normalsize $\Omega$} & {\normalsize ${\iota\Omega}$}\\
\hline\hline\hline
$\begin{bmatrix}
1 & 1 & 1\\
 & 1 & 1\\
 & & 1
\end{bmatrix}$
& 
$\begin{bmatrix}
0 & 0 & 0\\
 & 0 & 0\\
 & & 0
\end{bmatrix}$
\\ \hline\hline

$\begin{bmatrix}
1 & 1 & 0\\
 & 1 & 1\\
 & & 1
\end{bmatrix}$
& 
$\begin{bmatrix}
0 & 0 & 1\\
 & 0 & 0\\
 & & 0
\end{bmatrix}$
\\ \hline\hline

$\begin{bmatrix}
1 & 0 & 0\\
 & 1 & 0\\
 & & 1
\end{bmatrix}$
& 
$\begin{bmatrix}
0 & 1 & 1\\
 & 0 & 1\\
 & & 0
\end{bmatrix}$
\\ \hline\hline

$\begin{bmatrix}
1 & 0 & 1\\
 & 1 & 0\\
 & & 1
\end{bmatrix}$
& 
$\begin{bmatrix}
0 & 1 & 0\\
 & 0 & 1\\
 & & 0
\end{bmatrix}$
\\ \hline\hline

$\begin{bmatrix}
1 & 0 & 0\\
 & 1 & 1\\
 & & 1
\end{bmatrix}$
$\begin{bmatrix}
1 & 1 & 1\\
 & 0 & 0\\
 & & 1
\end{bmatrix}$
$\begin{bmatrix}
1 & 1 & 1\\
 & 1 & 1\\
 & & 0
\end{bmatrix}$
$\begin{bmatrix}
0 & 0 & 1\\
 & 1 & 0\\
 & & 1
\end{bmatrix}$
& 
$\begin{bmatrix}
0 & 0 & 1\\
 & 0 & 1\\
 & & 0
\end{bmatrix}$
$\begin{bmatrix}
0 & 1 & 0\\
 & 1 & 0\\
 & & 0
\end{bmatrix}$
$\begin{bmatrix}
1 & 0 & 0\\
 & 0 & 0\\
 & & 0
\end{bmatrix}$
$\begin{bmatrix}
0 & 1 & 0\\
 & 0 & 1\\
 & & 1
\end{bmatrix}$
\\ \hline\hline

$\begin{bmatrix}
0 & 1 & 1\\
 & 0 & 0\\
 & & 0
\end{bmatrix}$
$\begin{bmatrix}
0 & 0 & 0\\
 & 1 & 1\\
 & & 0
\end{bmatrix}$
$\begin{bmatrix}
0 & 0 & 0\\
 & 0 & 0\\
 & & 1
\end{bmatrix}$
$\begin{bmatrix}
1 & 1 & 0\\
 & 0 & 1\\
 & & 0
\end{bmatrix}$
& 
$\begin{bmatrix}
1 & 1 & 0\\
 & 1 & 0\\
 & & 1
\end{bmatrix}$
$\begin{bmatrix}
1 & 0 & 1\\
 & 0 & 1\\
 & & 1
\end{bmatrix}$
$\begin{bmatrix}
0 & 1 & 1\\
 & 1 & 1\\
 & & 1
\end{bmatrix}$
$\begin{bmatrix}
1 & 0 & 1\\
 & 1 & 0\\
 & & 0
\end{bmatrix}$
\\ \hline\hline

$\begin{bmatrix}
1 & 0 & 1\\
 & 1 & 1\\
 & & 1
\end{bmatrix}$
$\begin{bmatrix}
1 & 1 & 0\\
 & 0 & 0\\
 & & 1
\end{bmatrix}$
$\begin{bmatrix}
1 & 1 & 0\\
 & 1 & 1\\
 & & 0
\end{bmatrix}$
$\begin{bmatrix}
0 & 0 & 0\\
 & 1 & 0\\
 & & 1
\end{bmatrix}$
& 
$\begin{bmatrix}
0 & 0 & 0\\
 & 0 & 1\\
 & & 0
\end{bmatrix}$
$\begin{bmatrix}
0 & 1 & 1\\
 & 1 & 0\\
 & & 0
\end{bmatrix}$
$\begin{bmatrix}
1 & 0 & 1\\
 & 0 & 0\\
 & & 0
\end{bmatrix}$
$\begin{bmatrix}
0 & 1 & 1\\
 & 0 & 1\\
 & & 1
\end{bmatrix}$
\\ \hline\hline

$\begin{bmatrix}
0 & 1 & 0\\
 & 0 & 0\\
 & & 0
\end{bmatrix}$
$\begin{bmatrix}
0 & 0 & 1\\
 & 1 & 1\\
 & & 0
\end{bmatrix}$
$\begin{bmatrix}
0 & 0 & 1\\
 & 0 & 0\\
 & & 1
\end{bmatrix}$
$\begin{bmatrix}
1 & 1 & 1\\
 & 0 & 1\\
 & & 0
\end{bmatrix}$
& 
$\begin{bmatrix}
1 & 1 & 1\\
 & 1 & 0\\
 & & 1
\end{bmatrix}$
$\begin{bmatrix}
1 & 0 & 0\\
 & 0 & 1\\
 & & 1
\end{bmatrix}$
$\begin{bmatrix}
0 & 1 & 0\\
 & 1 & 1\\
 & & 1
\end{bmatrix}$
$\begin{bmatrix}
1 & 0 & 0\\
 & 1 & 0\\
 & & 0
\end{bmatrix}$
\\ \hline\hline

$\begin{bmatrix}
1 & 1 & 0\\
 & 0 & 0\\
 &  & 0
\end{bmatrix}$
$\begin{bmatrix}
0 & 0 & 0\\
 & 1 & 0\\
 &  & 0
\end{bmatrix}$
$\begin{bmatrix}
0 & 0 & 0\\
 & 0 & 1\\
 &  & 1
\end{bmatrix}$
&
$\begin{bmatrix}
1 & 1 & 1\\
 & 1 & 0\\
 &  & 0
\end{bmatrix}$
$\begin{bmatrix}
1 & 1 & 1\\
 & 0 & 1\\
 &  & 1
\end{bmatrix}$
$\begin{bmatrix}
0 & 0 & 1\\
 & 1 & 1\\
 &  & 1
\end{bmatrix}$
\\ 
$\begin{bmatrix}
0 & 1 & 1\\
 & 1 & 0\\
 &  & 1
\end{bmatrix}$
$\begin{bmatrix}
1 & 0 & 1\\
 & 1 & 1\\
 &  & 0
\end{bmatrix}$
$\begin{bmatrix}
1 & 0 & 1\\
 & 0 & 0\\
 &  & 1
\end{bmatrix}$
&
$\begin{bmatrix}
0 & 1 & 0\\
 & 0 & 0\\
 &  & 1
\end{bmatrix}$
$\begin{bmatrix}
1 & 0 & 0\\
 & 0 & 1\\
 &  & 0
\end{bmatrix}$
$\begin{bmatrix}
0 & 1 & 0\\
 & 1 & 1\\
 &  & 0
\end{bmatrix}$
\\ \hline\hline

$\begin{bmatrix}
1 & 1 & 1\\
 & 0 & 0\\
 &  & 0
\end{bmatrix}$
$\begin{bmatrix}
1 & 0 & 0\\
 & 1 & 1\\
 &  & 0
\end{bmatrix}$
$\begin{bmatrix}
0 & 0 & 1\\
 & 1 & 0\\
 &  & 0
\end{bmatrix}$
&
$\begin{bmatrix}
1 & 1 & 0\\
 & 1 & 0\\
 &  & 0
\end{bmatrix}$
$\begin{bmatrix}
1& 0 & 1\\
 & 0 & 1\\
 &  & 0
\end{bmatrix}$
$\begin{bmatrix}
1 & 1 & 0\\
 & 0 & 1\\
 &  & 1
\end{bmatrix}$
\\ 
$\begin{bmatrix}
1 & 0 & 0\\
 & 0 & 0\\
 &  & 1
\end{bmatrix}$
$\begin{bmatrix}
0 & 1 & 0\\
 & 1 & 0\\
 &  & 1
\end{bmatrix}$
$\begin{bmatrix}
0 & 0 & 1\\
 & 0 & 1\\
 &  & 1
\end{bmatrix}$
&
$\begin{bmatrix}
0 & 1 & 1\\
 & 1 & 1\\
 &  & 0
\end{bmatrix}$
$\begin{bmatrix}
0 & 1 & 1\\
 & 0 & 0\\
 &  & 1
\end{bmatrix}$
$\begin{bmatrix}
0 & 0 & 0\\
 & 1 & 1\\
 &  & 1
\end{bmatrix}$
\\ \hline
\end{tabular}
\caption{The table lists all the $\GG_3$-orbits (see \S\ref{sec:d4}). 
Each table entry represents a single $\GG_3$-orbit, which a collection of upper-triangular matrices comprising only 0s or 1s.}
\label{table}
\end{table}


\end{document}